\documentclass[a4paper,11pt]{article}
\usepackage{amsthm,amsmath,amssymb,mathrsfs,cancel}
\usepackage{graphicx,geometry}
\usepackage{float}
\usepackage{epstopdf}
\usepackage{epsfig}
\usepackage{hyperref,caption}
\usepackage{soul}
\usepackage{authblk}
\usepackage{ulem}
\usepackage{diagbox,array,multirow}

\theoremstyle{plain}
\newtheorem{Theorem}{Theorem}
\newtheorem{Lemma}[Theorem]{Lemma}
\newtheorem{Corollary}[Theorem]{Corollary}
\newtheorem{Proposition}[Theorem]{Proposition}
\newtheorem{Observation}[Theorem]{Observation}
\newtheorem{Claim}{Claim}

\theoremstyle{Definition}

\newtheorem{Conjecture}[Theorem]{Conjecture}
\theoremstyle{Remark}

\usepackage{authblk}
\usepackage{indentfirst}
\usepackage[symbol*]{footmisc}

\DefineFNsymbolsTM{otherfnsymbols}{%
  \textsection   \mathsection
  \textdagger    \dagger
\textasteriskcentered *
}%

\setfnsymbol{otherfnsymbols}

\usepackage{color}

\usepackage{setspace}
\newcommand{\proofend}{{\hfill$\Box$}}

\title{The burning number conjecture holds for trees of order $n$ with at most $\left\lfloor \sqrt{n-1}\right\rfloor$ degree-2 vertices}
	\author{
		Jiajun Ning\thanks{Email: jiajunning@126.com.},
		Xian'an Jin\thanks{Email: xajin@xmu.edu.cn.},
		Meiqiao Zhang\thanks{Corresponding Author.
			Email: meiqiaozhang@xmu.edu.cn and meiqiaozhang95@163.com.}
		\\
School of Mathematical Sciences, Xiamen University, P.R. China}
\date{}
\begin{document}
\onehalfspacing
\maketitle

\begin{abstract}
Inspired by the spread of information in social networks and graph-theoretic processes such as Firefighting and graph cleaning, Bonato, Janssen and Roshanbin introduced in 2016 the burning number $b(G)$ of any finite graph $G$. They conjectured that $b(G)\le \lceil n^\frac{1}{2}\rceil$ holds for all connected graphs $G$ of order $n$, and observed that it suffices to prove the conjecture for all trees. In 2024, Murakami confirmed the conjecture for trees without degree-2 vertices. In this paper, we prove that for all trees $T$ of order $n$ with $n_2$ degree-2 vertices, $$b(T)\le  \left\lceil \left(n+n_2-\left\lceil\sqrt{n+n_2+0.25}-1.5\right\rceil\right)^{\frac{1}{2}}\right\rceil.$$
Hence, the conjecture holds for all trees of order $n$ with at most $\left\lfloor \sqrt{n-1}\right\rfloor$ degree-2 vertices. 
\par
\noindent {\bf Keywords:} Burning Number Conjecture, Tree, Degree-2 vertex

\noindent {\bf MSC2020:} 05C05, 05C57
\end{abstract}
\section{Introduction}
In this paper, we assume all graphs are finite and simple.

Let $G$ be a connected graph. A \textbf{burning process} of $G$ is defined as follows.
Initially, all vertices in $G$ are unburned. In each round $r\ge 1$, we choose one unburned vertex $x_r$ in $G$ to be the \textbf{source} of round $r$, and burn the source $x_r$ and all the unburned vertices that are adjacent to some burned vertices in $G$. 
Clearly, once a vertex is burned in round $r$, all of its unburned neighbors in $G$ will be burned in round $r+1$. Once a vertex is burned, it remains burned. When all vertices in $G$ are burned, the burning process terminates. 
An example is shown in Figure \ref{P4}. 
\begin{figure}[H]
\centering
\includegraphics[width=4.5cm]{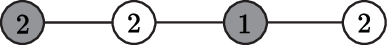}
\caption{Burning a path of order 4 in two rounds, where the vertices marked in grey are sources and the numbers on vertices represent the round they become burned}
\label{P4}
\end{figure}
If a burning process terminates in $k$ rounds, then $(x_1,x_2,\dots,x_k)$ is called a \textbf{burning sequence} for $G$ of \textbf{length} $k$. 
Obviously, any burning sequence for $G$ induces a burning process of $G$.
The \textbf{burning number} $b(G)$ of $G$ is the minimum length of a burning sequence for $G$.
A burning sequence for $G$ is called \textbf{optimal} if its length is $b(G)$.

This kind of process was first considered by Alon \cite{Elon} in 1992, motivated by a communication problem. Then in 2016, further inspired by the spread of information in social networks and graph-theoretic processes such as Firefighting and graph cleaning, Bonato, Janssen and Roshanbin \cite{Bonato} formally introduced the related concepts and raised the following conjecture,  which is known as the \textbf{Burning Number Conjecture}.
\begin{Conjecture}[Burning Number Conjecture \cite{Bonato}]
	Let $G$ be a connected graph of order $n$. Then $b(G)\le \lceil n^\frac{1}{2}\rceil$.
\end{Conjecture}
To support this conjecture,  Bonato, Janssen and Roshanbin \cite{Bonato} showed that the bound is sharp for all paths and cycles.
\begin{Proposition}[\cite{Bonato}]\label{PC}
	Let $G$ be the path or cycle of order $n$. Then $b(G)=\lceil n^\frac{1}{2}\rceil$.
\end{Proposition}

Further, Bonato, Janssen and Roshanbin \cite{Bonato}  simplified the conjecture by the following proposition, which indicates that if the conjecture holds for all trees, then it holds for all connected graphs.
\begin{Proposition}[\cite{Bonato}]\label{GT}
Let $G$ be a connected graph. Then
$$b(G)= \min\{b(T):  \text{T is a spanning tree of G}\}.$$
\end{Proposition}

From then on, a lot of researchers have focused on establishing upper bounds for the burning numbers of general graphs. Let $G$ be a connected graph of order $n$.  Bonato, Janssen and Roshanbin  \cite{Bonato} first showed in 2016 that $b(G)\le 2\lceil n^\frac{1}{2}\rceil-1$, and very soon, Land and Lu \cite{Land} improved the result into $b(G)\le \lceil \frac{-3+\sqrt{24n+33}}{4}\rceil$. In 2023, Bastide et al. \cite{Bastide} provided the best current bound that $b(G)\le \sqrt{\frac{4n}{3}}+1$, while in 2024, Norin and Turcotte \cite{Norin} proved the  asymptotic version of the Burning Number Conjecture.
However, the Burning Number Conjecture itself remains open.

Moreover, the conjecture has been extensively studied in the case of trees.
In 2018, Bessy et al. \cite{Bessy} characterized all the binary trees of depth $r$ with the burning number $r+1$, and showed that $b(T) \le \lceil \sqrt{n + n_2 + \frac{1}{4} }+ \frac{1}{2}\rceil$ for any tree $T$ of order $n$ with $n_2$ degree-2 vertices.
In 2024, Murakami \cite{Mur} improved this upper bound by showing that the conjecture holds for trees without degree-2 vertices.
\begin{Theorem}[\cite{Mur}]\label{Mur}
Let $T$ be a tree of order $n$ with $n_2$ degree-2 vertices. Then $b(T)\le \left\lceil (n+n_2)^\frac{1}{2}\right\rceil$.
\end{Theorem}
\noindent Within the same year, the conjecture was further confirmed for trees with a single degree-2 vertex in the Bachelor's thesis of van der Tol~\cite{Tol}, in which the author also noted the difficulty of confirming the conjecture for trees with slightly more degree-2 vertices.

In this paper, we shall refine Murakami's methods to  provide an enhanced upper bound on the burning numbers of trees, which may be lower than $\lceil n^{\frac12}\rceil$ for many trees of order $n$.
\begin{Theorem}\label{main1}
Let $T$ be a tree of order $n$ with $n_2$ degree-2 vertices. Then 
$$b(T)\le  \left\lceil \left(n+n_2-\left\lceil\sqrt{n+n_2+0.25}-1.5\right\rceil\right)^{\frac{1}{2}}\right\rceil.$$
\end{Theorem}

Moreover, Theorem~\ref{main1} confirms the Burning Number Conjecture for trees of order $n$ with at most $\left\lfloor \sqrt{n-1}\right\rfloor$ degree-2 vertices, which improves the results in \cite{Mur,Tol}.
	
\begin{Corollary}\label{main}
Let $T$ be a tree of order $n$ with at most $\left\lfloor \sqrt{n-1}\right\rfloor$ degree-2 vertices. Then $b(T)\le \lceil n^{\frac{1}{2}}\rceil$.
	\end{Corollary}
	
The proofs of Theorem~\ref{main1} and Corollary~\ref{main} will be given in Section~\ref{secmain}, while some preparatory results are developed in Section~\ref{secpre}.

\section{Preliminary results
\label{secpre}}
In this section, we establish some preliminary results for proving Theorem~\ref{main1} in the next section.

We first introduce some necessary terminologies and notations. For any positive integer $k$, let $[k]=\{1,2,\dots,k\}$.
For any graph $G$, denote by $V(G)$ the vertex set of $G$, $E(G)$ the edge set of $G$, and $|G|$ the order of $G$. For any $v\in V(G)$, denote by $d_G(v)$ the degree of $v$ in $G$, $N_G(v)$ the set of neighbors of $v$ in $G$, and $G-v$ the subgraph of $G$ obtained by deleting $v$ and all of its incident edges. If $d_G(v)\ge 2$, then $v$ is called an \textbf{internal vertex} of $G$. 
For any $V\subseteq V(G)$, denote by  $G[V]$ the subgraph of $G$ induced by $V$.
For any $uv\in E(G)$, denote by $G-uv$ the subgraph of $G$ obtained by deleting $uv$. 
For any tree $T$ and $uv\in E(T)$, let $T_v(uv)$ denote the connected component of $T-uv$ that contains $v$.

For convenience in the following proofs, we next show that, for any connected graph, one may allow the sources of some rounds (except the first) in a burning process to be empty.
\begin{Observation}\label{obsv}
Let $G$ be a connected graph with all vertices unburned and let $x_1\in V(G)$ be the source of round 1. In each round $r\ge 2$, let the source of round $r$ be empty or an unburned vertex $x_r$ in $G$, and burn $x_r$ (if it exists) and all the unburned vertices that are adjacent to some burned vertices in $G$. Once a vertex is burned, it remains burned. When all the vertices in $G$ are burned, this process terminates. Suppose that this process terminates in $k$ rounds. Then this process is identical to a burning process induced by some burning sequence $(y_1,y_2,\dots,y_k)$ for $G$ where for all $i\in [k]$, $y_i=x_i$ whenever $x_i$ exists.
\end{Observation}
\begin{proof}
Since $G$ is connected and this process terminates in $k$ rounds, for each $2\le r\le k$, there is some $s_r\in V(G)$ that becomes burned in round $r$. In other words, $s_r$ is unburned in round $r-1$ and burned in round $r$. Note that $s_r$ is possibly $x_r$ if $x_r$ exists.

For each $i\in [k]$, let $y_i=x_i$ if $x_i$ exists, and let $y_i=s_i$ otherwise. Then it can be easily verified that $(y_1,y_2,\dots,y_k)$ is a burning sequence for $G$ that induces a burning process identical to the given process.
\end{proof}

Also, we present a lemma, which generalizes Lemma 2 from~\cite{Mur} by decomposing a tree into several smaller components through the removal of a single vertex.

\begin{Lemma}\label{sep}
Let $T$ be a tree of order $n\ge 3$. Then for any real number $p\in [1, n-1)$, there exists a vertex $v$ in $T$ with $N_T(v)=\{v_1,v_2,\dots,v_k\}$, where $k\ge 2$, such that $|T_v(vv_k)|> p$ and $|T_{v_i}(vv_i)|\leq p$ for all $i\in [k-1]$. 
\end{Lemma}
\begin{proof}
Let $v$ be a vertex in $T$ with $N_T(v)=\{v_1,v_2,\dots,v_k\}$ such that $v_k$ is a leaf of $T$.
Then $|T_{v}(vv_k)|=n-1>p\ge 1$, which implies that $v$ is not a leaf of $T$ and $k\ge 2$. If $|T_{v_i}(vv_i)|\leq p$ for all $i\in[k-1]$, then the consequence holds. Otherwise, without loss of generality, we suppose $|T_{v_1}(vv_1)|> p\ge 1$. \par 
Since $|T_{v_1}(vv_1)|\ge 2$, $v_1$ has at least two neighbors in $T$. Suppose that $N_T(v_1)=\{y_1,y_2,\dots,y_j\}$, where $j\ge 2$ and $y_j=v$. Then $|T_{v_1}(v_1y_j)|=|T_{v_1}(vv_1)|> p$. If $|T_{y_i}(v_1y_i)|\leq p$ for all $i\in[j-1]$, then the consequence holds. Otherwise, we continue this process in the same way. Since $T$ is a tree of a finite order $n$ and $p\ge 1$, the process must terminate. Hence the consequence holds.
\end{proof}

For any tree $T$ and a degree-2 vertex $w\in V(T)$ with $N_T(w)=\{w_1,w_2\}$, the \textbf{split} of $T$ by $w$ is the forest obtained from $T$ by replacing $w$ with two new vertices, one adjacent to $w_1$ and the other adjacent to $w_2$.

Now we provide an upper bound on the number of internal vertices in a tree in terms of its order and the number of degree-2 vertices.

\begin{Lemma}\label{sep2}
Let $T$ be a tree of order $n$ with $n_2$ degree-2 vertices. Then $T$ contains at most $\left\lfloor\frac{n_2+n-2}{2}\right\rfloor$ internal vertices.
\end{Lemma}
\begin{proof}
We split $T$ into $n_2+1$ connected components $T_1,T_2,\dots,T_{n_2+1}$ by the $n_2$ degree-2 vertices. Then each component $T_i$ ($1\leq i\leq n_2+1$) is a tree without degree-2 vertices.

For each $i\in[n_2+1]$, let $l_i$ and $I_i$ be the numbers of leaves and internal vertices in $T_i$ respectively. Then $|T_i|=l_i+I_i$, and $l_i\geq I_i+2$ by the Handshaking Lemma. Let $I$ be the number of internal vertices in $T$. Then $I=n_2+\sum_{i=1}^{n_2+1}I_i$. Moreover,
$$|T|=n=\sum_{i=1}^{n_2+1}(l_i+I_i)-n_2\geq \sum_{i=1}^{n_2+1}(2I_i+2)-n_2=2(I+1)-n_2.$$ 
It implies that $I\le \frac{n_2+n-2}{2}$. Since $I$ is an integer, the result is proven.
\end{proof}

Let $T$ be a tree and $w\in V(T)$ with $N_T(w)=\{w_1,w_2,\dots,w_q\}$, where $q\ge 2$. 
Suppose that $N_T(w)$ contains $p$ leaves of $T$, say $w_1,w_2,\dots,w_p$ if $p>0$.
Then a tree $T'$ obtained by \textbf{smoothing} $w$ in $T$ is a tree formed from $T$ by deleting $w$ and performing the following operation:
\begin{enumerate}
\item if $p\le 2$, then form a path $w_1w_{3}w_{4}\dots w_qw_2$;
\item if $p\ge3$, then delete $w_3,w_4,\dots,w_p$ and form a path $w_1w_{p+1}w_{p+2}\dots w_qw_2$.
\end{enumerate}
See Figures \ref{Smooth1} and \ref{Smooth2} as examples.
Then it is clear to see that $V(T)\setminus V(T')$ consists of $w$ and some leaves in $T$, and if no vertex in $T$ other than $w$ has degree two, then  $T'$ contains no degree-2 vertices.
\begin{figure}[ht]
\centering
\includegraphics[width=13.5cm]{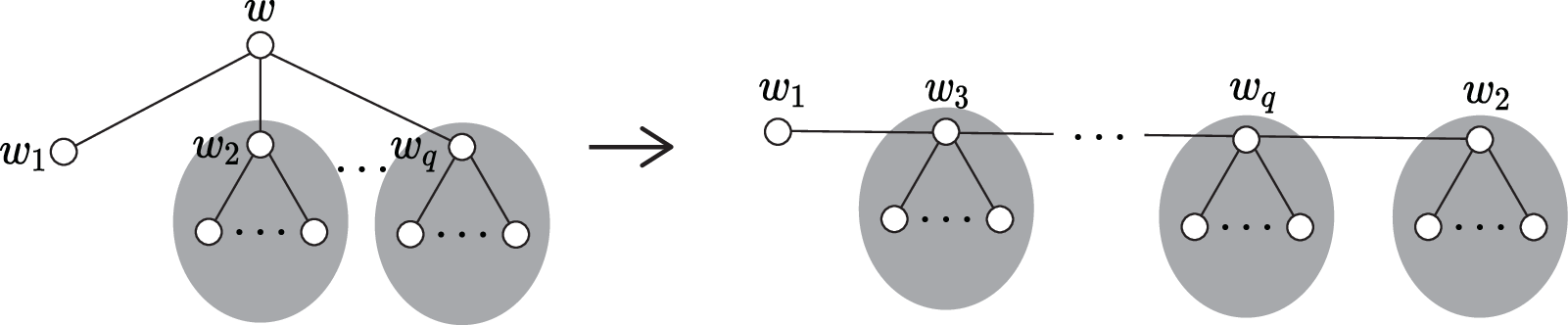}

\vspace{0cm}
{}\hfill \hspace{0cm} (a) $T$  \hspace{7.2cm}  (b) $T'$  \hfill {}
\caption{An example for $T$ and $T'$ with $p=1$}
\label{Smooth1}
\end{figure}
\begin{figure}[ht]
\centering
\includegraphics[width=13cm]{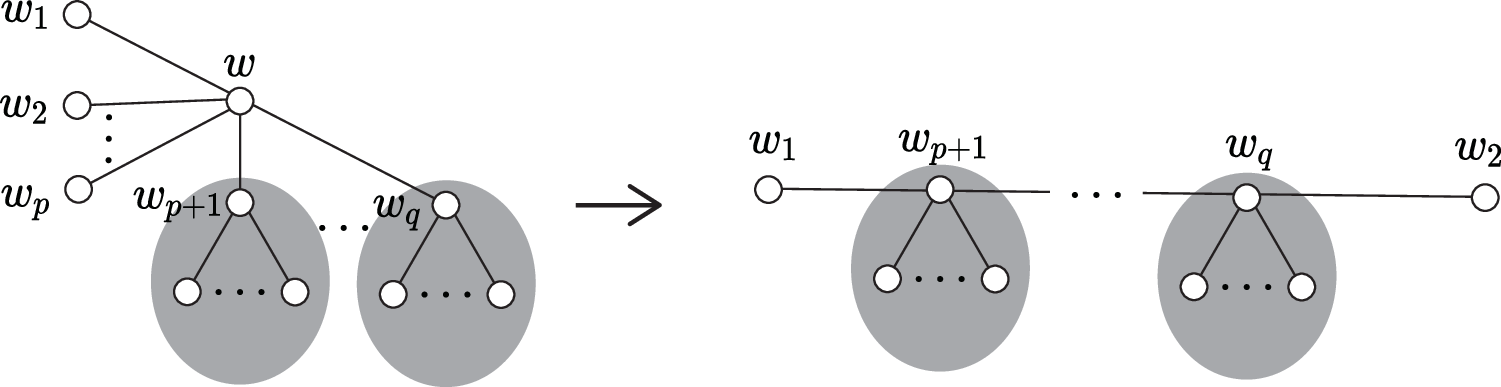}

\vspace{0cm}
{}\hfill \hspace{-0.5cm} (a) $T$  \hspace{5.5cm}  (b) $T'$  \hfill {}    	
\caption{An example for $T$ and $T'$ with $p\ge2$}
\label{Smooth2}
\end{figure}

To conclude this section, we show that burning sequences for some special trees can be obtained recursively through removing leaves and smoothing operations.
\begin{Lemma}\label{sep3}
Suppose that $T$ is a tree with $u\in V(T)$, where $d_T(u)\ge 3$ and $u$ is adjacent to a leaf $v$. Let $T'$ be a tree obtained by smoothing $u$ in $T-v$. Then $b(T)\le b(T')+1$. In particular, $T$ has a burning sequence of length not greater than $b(T')+1$ beginning with $v$.
\end{Lemma}
\begin{proof}
Let $t=b(T')$. Then $t\ge 2$ as $T'$ is obtained by smoothing $u$ in $T-v$ and $d_T(u)\ge 3$.

Let $(s_1,s_2,\dots,s_t)$ be an optimal burning sequence for $T'$. For each $w\in V(T')$, let $\ell'(w)$ be the round in which $w$ becomes burned in the corresponding burning process. Then $\ell'(w)\le t$ for all $w\in V(T')$.

Now it suffices to build a burning sequence for $T$ of length not greater than $t+1$ beginning with $v$.
Following the context of Observation~\ref{obsv}, let $v$ be the source of round 1,
$s_i$ be the source of round $i+1$ for $1\le i\le t$ if $s_i$ is unburned in round $i$, and the sources of all the other rounds be empty.
For each $w\in V(T)$, denote by $\ell(w)$ the round $w$ becomes burned in this process. Then by Observation~\ref{obsv}, it remains to show that $\ell(w)\le t+1$ for all $w\in V(T)$.

Note that $\ell(v)=1$, $\ell(u)=2$, $\ell(w)=3\le t+1$ for all the leaves $w$ in $V(T)\setminus V(T')$, and for all $i\in [t]$, 
\begin{equation}\label{ineq6}
\ell(s_i)\le \ell'(s_i)+1\le t+1.
\end{equation}
In the following, we shall show that $\ell(w)\le \ell'(w)+1$ for all $w\in (V(T)\cap V(T'))\setminus\{s_1,s_2,\dots,s_t\}$.

As shown in Figure \ref{Smooth3}, each vertex in $V(T)\cap V(T')$
belongs to a unique $T_{x}(ux)$ for some $x\in N_T(u)\cap V(T')$, where the structure of $T_{x}(ux)$ is the same in $T$ and $T'$,  and the only vertex in $T_{x}(ux)$ adjacent to some vertex in $V(T)\setminus V(T_{x}(ux))$ and $V(T')\setminus V(T_{x}(ux))$ is $x$. As a result, each vertex in $V(T_{x}(ux))\setminus\{x,s_1,s_2,\dots,s_t\}$ becomes burned only due to adjacent burned vertices in $V(T_{x}(ux))$. Then by (\ref{ineq6}), as long as each $x$ in $N_T(u)\cap V(T')$ satisfies $\ell(x)\le \ell'(x)+1$, we have
$\ell(w)\le \ell'(w)+1$ for all vertices $w$ in $(V(T)\cap V(T'))\setminus\{s_1,s_2,\dots,s_t\}$.
\begin{figure}[H]
\centering
\includegraphics[width=13cm]{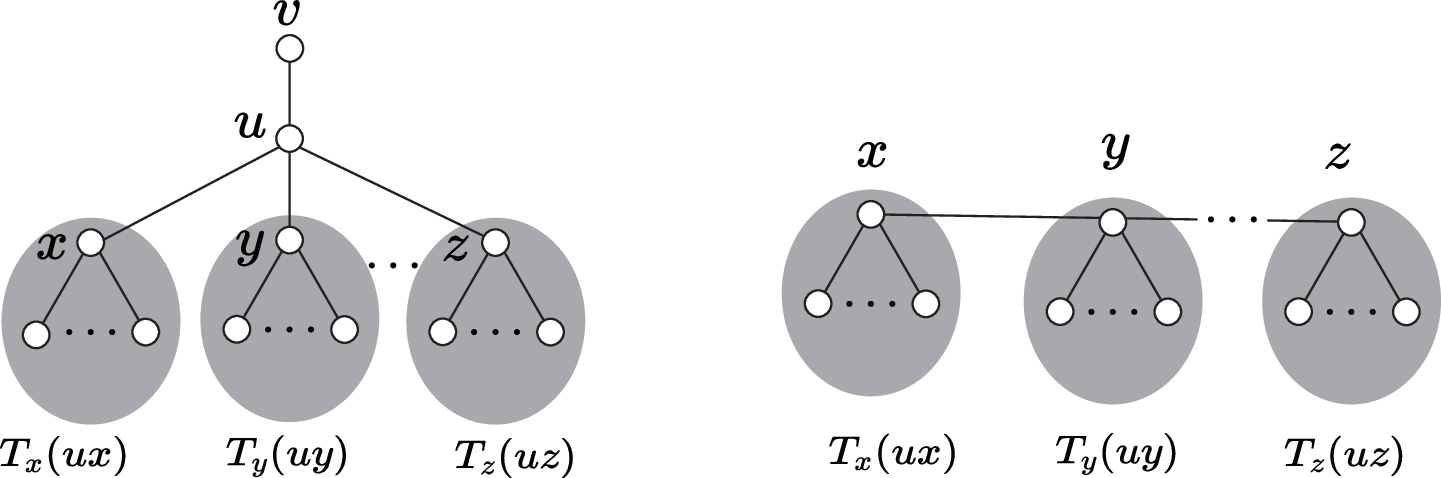}

\vspace{0cm}
{}\hfill \hspace{-0.2cm} (a) $T$  \hspace{6.2cm}  (b) $T'$  \hfill {}    	
\caption{An example for $T$ and $T'$}
\label{Smooth3}
\end{figure}

Since $\ell(u)=2$, it is clear that
$\ell(x)=\min \{\ell'(x)+1, 3\}$ for all $x\in N_T(u)\cap V(T')$. The consequence holds.
\end{proof}

Observation~\ref{obsv}, Lemmas~\ref{sep},~\ref{sep2} and~\ref{sep3} will be applied in the proof of Proposition~\ref{00}, which is a crucial step into proving Theorem~\ref{main1}.

\section{Proofs of main results}
\label{secmain}
In this section, we shall prove Theorem~\ref{main1} and Corollary~\ref{main}.

We first show the initial step for proving Theorem~\ref{main1} by the following proposition.

\begin{Proposition}\label{00}
Let $T$ be a tree of order $n\ge m(m+1)+1$ without degree-2 vertices, where $m$ is a non-negative integer. Then $b(T)\le  \lceil (n-m)^{\frac{1}{2}}\rceil$.	
\end{Proposition}
\begin{proof}
If $m=0$ and $n\ge 1$, then the consequence follows from Theorem \ref{Mur}.

Now suppose that $b(T')\le  \lceil (n'-m')^{\frac{1}{2}}\rceil$ holds for all trees $T'$ of order $n'\ge m'(m'+1)+1$ without degree-2 vertices, where $m'<m$ is a non-negative integer.
Then we are going to show that $b(T)\le  \lceil (n-m)^{\frac{1}{2}}\rceil$ holds for any tree $T$ of order $n\ge m(m+1)+1$ and without degree-2 vertices. Here $m\ge 1$ and 
\begin{equation}\label{ineq3}
m= \lceil (m^2+1)^{\frac{1}{2}}\rceil-1\le \lceil (n-m)^{\frac{1}{2}}\rceil-1,
\end{equation}
where the last inequality is tight when $n= m(m+1)+1$.

If $n=m(m+1)+1$, then  $m<\sqrt{m^2+m+1}=\sqrt{n} < m+1$, implying that $\lceil n^{\frac{1}{2}}\rceil=m+1$. Then by Theorem \ref{Mur} and (\ref{ineq3}), for any tree $T$ of order $n=m(m+1)+1$ without degree-2 vertices, 
$$b(T)\le  \lceil n^{\frac{1}{2}}\rceil=m+1=\lceil (n-m)^{\frac{1}{2}}\rceil.$$ 
The consequence holds.

In the following, suppose that $T$ is a tree of order $n\ge m(m+1)+2$ without degree-2 vertices. Then $n\ge 4$. Let $p=2\lceil (n-m)^{\frac{1}{2}}\rceil-1.5$. Then $p\in [1,n-1)$ as $m\ge 1$. According to Lemma \ref{sep}, there exists a vertex $v$ in $T$ with $N_T(v)=\{v_1,v_2,\dots,v_k\}$, where $k\ge 2$, such that $|T_v(vv_k)|>p$ and $|T_{v_i}(vv_i)|\le p$ for all $i\in [k-1]$. 
For each $i\in [k]$, let  $T_i=T[V(T_{v_i}(vv_i))\cup \{v\}]$. 

In order to find a burning sequence for $T$, we have the following two claims to analyze $T_1,T_2,\dots, T_k$ separately.

\begin{Claim}\label{clm1}
For each $T_i$ with $i\in [k-1]$, following the context of Observation~\ref{obsv}, let $v$ be the source of round 1 and let the sources of all the other rounds be empty. Then this process of $T_i$ terminates in at most  $\lceil (n-m)^{\frac{1}{2}}\rceil$ rounds.
\end{Claim}
\noindent\textit{Proof.}
Since $T$ contains no degree-2 vertices, $T_i$ also contains no degree-2 vertices. Then by Lemma \ref{sep2}, the number of internal vertices in $T_i$ is at most 
$$\left\lfloor \frac{0+\lfloor p\rfloor+1-2}{2}\right\rfloor=\left\lfloor \lceil (n-m)^{\frac{1}{2}}\rceil-1.5\right\rfloor= \lceil (n-m)^{\frac{1}{2}}\rceil-2.$$

For burning $T_i$, let $v$ be the source of  round 1. Then in each of the following rounds, at least one internal vertex in $T_i$ will be burned by adjacent burned vertices. Since $T_i$ contains at most $\lceil (n-m)^{\frac{1}{2}}\rceil-2$ internal vertices and all leaves are adjacent to internal vertices, all vertices in $T_i$ will be burned via adjacency within $\lceil (n-m)^{\frac{1}{2}}\rceil$ rounds. 

Claim~\ref{clm1} holds.
{\hfill $\natural $}

\begin{Claim}\label{clm2}
$T_k$ has a burning sequence $(v,x_1,x_2,\dots,x_t)$, where $t\le \lceil (n-m)^\frac{1}{2}\rceil-1$ and $x_i\in V(T_{v_k}(vv_k))$ for all $i\in [t]$.
\end{Claim}
\noindent\textit{Proof.}
Since $T$ contains no degree-2 vertices, no vertex in $T_{v_k}(vv_k)$ has degree two, except possibly $v_k$.
Moreover, $v_k$ has degree either zero or at least two in $T_{v_k}(vv_k)$.
For the former case, $|T_{v_k}(vv_k)|=1$, and by (\ref{ineq3}), 
$$b(T_{v_k}(vv_k))=1\le m\le \lceil (n-m)^\frac{1}{2}\rceil-1.$$ 
Then the claim obviously holds.

In the following, suppose that $v_k$ has degree at least two in $T_{v_k}(vv_k)$. 
Let $T'$ be a tree obtained by smoothing $v_k$ in $T_{v_k}(vv_k)$. Then $T'$ contains no degree-2 vertices and
\begin{equation}\label{ineq1}
|T'|\le |T_{v_k}(vv_k)|-1\le \lfloor n-p\rfloor-1= n- 2\lceil (n-m)^{\frac{1}{2}}\rceil\le 
(\lceil (n-m)^{\frac{1}{2}}\rceil -1)^2 +(m-1).
\end{equation}

Below, we shall show that $b(T')\le \lceil (n-m)^{\frac{1}{2}}\rceil-1$. Then applying Lemma~\ref{sep3} to $T_k$, $v_k$ and the leaf $v$, the claim is true.

If $|T'|\le m^2$,
then by Theorem \ref{Mur} and (\ref{ineq3}),
$$b(T')\le  \lceil (m^2)^{\frac{1}{2}}\rceil = m\le \lceil (n-m)^{\frac{1}{2}}\rceil-1.$$

Otherwise, $|T'|\ge m^2+1$. Then $|T'|\ge (m-1)m+1$ obviously holds. By the assumption and (\ref{ineq1}),
$$b(T')\le \lceil (|T'|-(m-1))^{\frac12} \rceil 
\le \lceil (n-m)^{\frac{1}{2}}\rceil-1.$$

Hence Claim~\ref{clm2} holds.
{\hfill $\natural $}

Now we are ready to burn $T$. Following the context of Observation~\ref{obsv}, let $v$ be the source of round 1, let $x_i$ be the source of round $i+1$ for $1\le i\le t$, and let the source of each following round be empty. Claims~\ref{clm1} and~\ref{clm2} together tell us that this process terminates within $\lceil (n-m)^{\frac{1}{2}}\rceil$ rounds. Thus by Observation~\ref{obsv}, there exists a burning sequence for $T$ of length not greater than $\lceil (n-m)^{\frac{1}{2}}\rceil$.
The result is proven.
\end{proof}
\noindent\textbf{Remark.} Note that the condition $n\ge m(m+1)+1$ in Proposition \ref{00} is the best possible. See the tree $T$ of order 6 without degree-2 vertices in Figure~\ref{T6} for an example. For $n=6$ and $m=2$ with $n=m(m+1)$, it can be easily verified that $b(T)=3>\lceil (6-2)^{\frac12} \rceil$. 
\begin{figure}[ht]
\centering
\includegraphics[width=4.5cm]{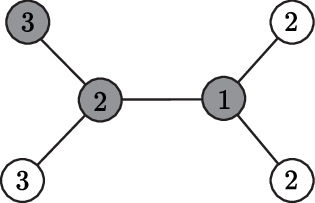}
\caption{A tree $T$ of order 6, without degree-2 vertices and $b(T)=3$}
\label{T6}
\end{figure}

Now we prove Theorem~\ref{main1} via the proposition below.

\begin{Proposition}[\cite{Bonato}]\label{isometric}
	Let $T_0$ be a subtree of a tree $T$. Then $b(T_0)\le b(T)$.
\end{Proposition}

\noindent\textit{Proof of Theorem~\ref{main1}.}
Let  $m=\left\lceil\sqrt{(n+n_2)+0.25}-1.5\right\rceil\ge 0$. Then $n+n_2\ge m(m+1)+1$.

Let $T_1$ be the tree obtained from $T$ by adding a leaf to each of the degree-2 vertices in $T$ such that no vertex remains degree 2. Then $|T_1|=n+n_2$. By Propositions \ref{00} and \ref{isometric}, 
$$b(T)\le b(T_1)\le \left\lceil \left((n+n_2)-\left\lceil\sqrt{n+n_2+0.25}-1.5\right\rceil\right)^{\frac{1}{2}}\right\rceil.$$
\proofend

\vspace{0.3cm}
The proof of Corollary \ref{main} is as follows.

\vspace{0.3cm}
\noindent \textit{Proof of Corollary \ref{main}.} 
If $n=1$, then the consequence holds. 
For $n\ge 2$, let $n_2$ be the number of degree-2 vertices in $T$. We shall show that for all integers $n_2\le \left\lfloor \sqrt{n-1}\right\rfloor$, 
\begin{equation}\label{eq1}
\left\lceil \left(n+n_2-\left\lceil\sqrt{n+n_2+0.25}-1.5\right\rceil\right)^{\frac{1}{2}}\right\rceil\le \lceil n^{\frac12}\rceil,
\end{equation}
by which and Theorem~\ref{main1} the consequence holds. 

Let 
$$
f_n(x)= \left(n+x-\sqrt{n+x+0.25}+1.5\right)^{\frac{1}{2}},
$$
where $0\le x\le n-2$. 
Then $f_n(x)>0$ and $f_n(x)$ is monotonically increasing as
$$\frac{d\left(f_n(x)\right)}{dx}=\frac{1-\frac{1}{2\sqrt{n+x+0.25}}}{2f_n(x)}>0.$$

Let $x_0=\sqrt{n-1}-1$. Then $f_n(x_0)=\sqrt{n}$.
Since $f_n(x)$ is monotonically increasing, we have $f_n(x)\le \sqrt{n}$ for all $x\le x_0$, 
which implies that $x\le \sqrt{n+x+0.25}-1.5$ and 
$$\lceil x\rceil \le\left\lceil\sqrt{n+x+0.25}-1.5\right\rceil\le \left\lceil\sqrt{n+\lceil x\rceil+0.25}-1.5\right\rceil.$$ 
Thus for all $x\le x_0$, we have
$$\left\lceil \left(n+\lceil x\rceil-\left\lceil\sqrt{n+\lceil x\rceil+0.25}-1.5\right\rceil\right)^{\frac{1}{2}}\right\rceil  \le \lceil  n^{\frac12}\rceil.$$
Hence (\ref{eq1}) holds for all integers $n_2\le \lfloor x_0\rfloor =\left\lfloor \sqrt{n-1}-1\right\rfloor$.

Now it remains to consider the case when $n_2=\left\lfloor \sqrt{n-1}\right\rfloor\ge 1$.
Note that in this case,
$\sqrt{(n_2+0.5)^2+1}\in (n_2+0.5,n_2+1]$, 
which means $\left\lceil \sqrt{(n_2+0.5)^2+1}-0.5\right\rceil=n_2+1$. Then 
$$n_2=\left\lceil \sqrt{(n_2+0.5)^2+1}-0.5 \right\rceil-1\le \left\lceil\sqrt{n+n_2+0.25}-1.5\right\rceil,$$
which implies (\ref{eq1}) holds.
\proofend

As applications, the following corollaries are direct.

\begin{Corollary}\label{fullbinary}
	Let $T$ be a full binary tree of order $n$. Then $b(T)\le \lceil n^{\frac{1}{2}}\rceil$.
\end{Corollary}
\begin{proof}
Note that in a full binary tree, all vertices have either 2 or 0 children. As a result, a full binary tree has at most one degree-2 vertex, which is the root. Then the result follows from Corollary \ref{main}.
\end{proof}

\begin{Corollary}
	Let $G$ be a graph of order $n$. If $G$ has a spanning tree $T$ such that $T$ contains at most $\left\lfloor \sqrt{n-1}\right\rfloor$ degree-2 vertices, then $b(G)\leq \lceil n^{\frac{1}{2}}\rceil$.
\end{Corollary}

\begin{proof}
The consequence follows from Corollary \ref{main} and Proposition \ref{GT}.
\end{proof}

\section{Concluding remarks}
Theorem \ref{main1} establishes an upper bound of the burning number for a tree $T$ of order $n$ with $n_2$ degree-2 vertices as $\left\lceil \left(n+n_2-\left\lceil\sqrt{n+n_2+0.25}-1.5\right\rceil\right)^{\frac{1}{2}}\right\rceil$. 
In particular, the Burning Number Conjecture holds for trees of order $n$ with at most $\left\lfloor \sqrt{n-1}\right\rfloor$  degree-2 vertices. 
Besides, for trees containing no degree-2 vertices, we obtain an upper bound of the burning number which may be lower than $\lceil n^{\frac{1}{2}}\rceil$.  
However, as also noted by Murakami in~\cite{Mur} for Theorem~\ref{Mur}, Theorem~\ref{main1} might not be better for trees of large orders than the result from an unpublished manuscript [S. Das, S.S. Islam, R.M. Mitra, S. Paul, Burning a binary tree and its generalization,  arXiv:2308.02825]. In their work, they claim to have developed an algorithm that establishes the bound $b(T)\le \lceil(n+n_2+8)^{\frac{1}{2}})\rceil-1$ for any tree $T$ of order $n\ge 50$ with $n_2$ degree-2 vertices.  
We have not verified the correctness of their result.

 \section*{Acknowledgements}

The authors wish to thank Jun Ge for providing valuable suggestions and comments that have significantly contributed to improving this manuscript.
This research is supported by National Natural Science Foundation of China (No. 12171402).  

\end{document}